\documentclass{amsart}

\usepackage[utf8]{inputenc}
\usepackage[T1]{fontenc}
\usepackage{amsmath,amsthm,amsfonts,amssymb}
\usepackage{microtype}
\usepackage[hidelinks]{hyperref}
\newtheorem{theorem}{Theorem}[section]
\newtheorem{lemma}[theorem]{Lemma}
\newtheorem{proposition}[theorem]{Proposition}
\newtheorem{corollary}[theorem]{Corollary}
\numberwithin{equation}{section}
\newtheorem{question}[theorem]{Question}
\theoremstyle{remark}

\theoremstyle{plain}

\begin{document}

\title[The lower envelope of ultraspherical polynomials]
{The lower envelope of ultraspherical polynomials}
\author{K. Castillo}
\address{CMUC, Department of Mathematics, University of Coimbra,
3000-143 Coimbra, Portugal}
\email{kenier@mat.uc.pt}
\author{S. Sadigova}
\address{Institute of Mathematics, The Ministry of Science and Education,
Baku, Azerbaijan}
\email{s\_sadigova@mail.ru}
\hypersetup{pdfauthor={K. Castillo and S. Sadigova}}
\date{September 5, 2026}

\begin{abstract}
Let $R_n^{(a,b)}=P_n^{(a,b)}/P_n^{(a,b)}(1)$ be the normalised Jacobi
polynomials. For $\alpha\ge0$, put $\xi_0=-1$ and let $\xi_n$ be the
largest zero of $R_n^{(\alpha+1,\alpha)}$ for $n\ge1$. We prove that
\[
R_k^{(\alpha,\alpha)}(x)
=\min_{n\ge0}R_n^{(\alpha,\alpha)}(x),
\quad \xi_{k-1}\le x\le\xi_k,
\]
for every $k\ge1$. This determines the pointwise lower envelope of the
normalised ultraspherical polynomials and answers Question~2 posed by
F.~M. de Oliveira Filho in his 2009 thesis. For $\alpha>0$ and $-1<x<1$,
we also determine all minimising degrees. The sequence is non-increasing
up to each such degree, which answers his Question~1 in this range.
A Legendre example gives a negative answer when $\alpha=0$.
\end{abstract}

\maketitle

\section{Introduction}\label{intro}

For $a,b>-1$, write
\begin{equation}\label{eq:jacobi-normalisation}
R_n^{(a,b)}(x)=\frac{P_n^{(a,b)}(x)}{P_n^{(a,b)}(1)},
\quad n\ge0,
\end{equation}
where $P_n^{(a,b)}$ is the Jacobi polynomial of degree $n$.
We determine the pointwise lower envelope of
$(R_n^{(\alpha,\alpha)})_{n\ge0}$ on $[-1,1]$ for $\alpha\ge0$.
The description uses consecutive intervals determined by the largest
zeros of a neighbouring Jacobi family.

Such minima enter bounds for measurable chromatic numbers derived through
harmonic analysis. Bachoc, Nebe, de Oliveira Filho, and Vallentin expressed their
extension of the Lov\'asz theta function for spherical distance graphs in
terms of a minimum of normalised ultraspherical polynomials
\cite[Theorem~6.2]{BachocNebeOliveiraVallentin2009}.

Fix $\alpha\ge0$. Put $\xi_0=-1$, and let $\xi_n$ be the largest zero
of $R_n^{(\alpha+1,\alpha)}$ for $n\ge1$. Interlacing and the density
of Jacobi zeros give
\begin{equation}\label{eq:envelope-breakpoints}
-1=\xi_0<\xi_1<\xi_2<\cdots<1,
\quad \lim_{n\to\infty}\xi_n=1;
\end{equation}
see \cite[Theorems~3.3.1, 3.3.2, and~6.1.1]{Szego1975}.
In Question~2 of his thesis, de Oliveira Filho asked whether degree $k$
attains the minimum throughout $[\xi_{k-1},\xi_k]$
\cite[Section~3.5d, p.~47]{OliveiraFilho2009}.
Here we have shifted the index in his formulation by one.
Our main result gives an affirmative answer.

\begin{theorem}\label{thm:main}
Let $\alpha\ge0$ and $k\ge1$. For every
$x\in[\xi_{k-1},\xi_k]$,
\begin{equation}\label{eq:main-envelope}
R_k^{(\alpha,\alpha)}(x)
=\min_{n\ge0}R_n^{(\alpha,\alpha)}(x).
\end{equation}
\end{theorem}

The closed intervals include ties: at $x=\xi_k$, both degrees $k$ and
$k+1$ attain the minimum, for every $k\ge1$. Additional minimising
degrees can occur, as discussed in Section~\ref{sec:minimising-indices}.
Together with \eqref{eq:envelope-breakpoints}, the theorem covers every
$x\in[-1,1)$; at $x=1$, all the normalised polynomials equal one.

The minimum comparison was known when $x$ is the largest zero of
$P_{k-1}^{(\alpha+1,\alpha+1)}$, for $k\ge2$
\cite[Proposition~7.1]{BachocNebeOliveiraVallentin2009}.
De Oliveira Filho also proved that, at that point, degree $k$ is the
unique global minimiser and the sequence of values decreases strictly
up to degree $k$ \cite[Theorem~3.8 and its proof]{OliveiraFilho2009}.
Koornwinder recalled this case in his discussion of the related
Question~1 \cite[Topic~11, Remark~2, p.~21]{Koornwinder2010}.
Theorem~\ref{thm:main} establishes the minimum comparison on the full
interval $[\xi_{k-1},\xi_k]$.

The theorem also implies that, at each fixed $-1<x<1$, the sequence
$(R_n^{(\alpha,\alpha)}(x))_{n\ge0}$ is non-increasing up to its least
minimising degree. We prove this consequence as
Corollary~\ref{cor:least-minimiser} in Section~\ref{sec:minimising-indices}.
Question~1 of de Oliveira Filho, reproduced by Koornwinder
\cite[Topic~11, p.~21]{Koornwinder2010}, asks whether the same conclusion
holds for an arbitrary minimising degree. A Legendre example gives a
negative answer when $\alpha=0$. For $\alpha>0$,
Proposition~\ref{prop:positive-minimisers} shows that the minimum is
unique between successive breakpoints and is attained at exactly two
consecutive degrees at each breakpoint. This gives an affirmative answer
to Question~1 for positive parameters. Here ``decreasing'' is understood
as non-increasing.

The proof begins in Section~\ref{sec:auxiliary} with a representation
of the later ultraspherical polynomials in terms of two auxiliary
recurrence solutions. A positive connection formula reduces the required
upper bounds for these solutions to the parameter $\lambda=1/2$, where
$\lambda=\alpha+1/2$. Section~\ref{sec:auxiliary-bound} proves the bound
by an estimate for a quadratic form. An estimate for the largest Jacobi zero
then yields the theorem for $k\ge3$ in Section~\ref{sec:largest-zero}.
Section~\ref{sec:first-intervals} treats the first two intervals and
completes the proof. In Section~\ref{sec:minimising-indices}, we determine
the equality cases for $\alpha>0$ and settle Question~1.

\section{Auxiliary solutions and positive connections}\label{sec:auxiliary}

Throughout the proof, set $\lambda=\alpha+1/2$ and write
$U_n=R_n^{(\alpha,\alpha)}$ and
$Q_n=R_n^{(\alpha+1,\alpha)}$ for $n\ge0$.
We repeatedly use the contiguous relations
\begin{subequations}\label{eq:contiguous-relations}
\begin{align}
(n+\alpha+1)(1-x)Q_n(x)
&=(\alpha+1)\bigl(U_n(x)-U_{n+1}(x)\bigr),
\label{eq:contiguous-difference}\\[7pt]
U_n(x)
&=\frac{(n+\alpha+1)(n+2\alpha+1)Q_n(x)
-n(n+\alpha)Q_{n-1}(x)}
{(\alpha+1)(2n+2\alpha+1)}.
\label{eq:contiguous-representation}
\end{align}
\end{subequations}
Identity \eqref{eq:contiguous-difference} holds for $n\ge0$, and
\eqref{eq:contiguous-representation} holds for $n\ge1$.
After normalisation, the first follows from \cite[(4.5.4)]{Szego1975};
the second follows from Jacobi symmetry and
\cite[Table~18.6.1 and (18.9.5)]{DLMF}.
We also use the normalised ultraspherical
recurrence
\begin{equation}\label{eq:ultraspherical-recurrence}
(n+2\lambda)U_{n+1}(x)
=2(n+\lambda)xU_n(x)-nU_{n-1}(x).
\end{equation}
Here $n\ge1$, $U_0(x)=1$, and $U_1(x)=x$.
This is \cite[(4.7.17)]{Szego1975} after normalisation.

\subsection{A representation for higher degrees}

For $c>0$, define $W_r=W_r^{(c,\lambda)}$ by
\begin{equation}\label{eq:W-recurrence}\begin{aligned}
W_{-1}(x)&=W_0(x)=1,\\[7pt]
(c+r+2\lambda)W_{r+1}(x)
&=2(c+r+\lambda)xW_r(x)-(c+r)W_{r-1}(x).
\end{aligned}\end{equation}
The recurrence is used for $r\ge0$.
In particular, $W_r^{(c,\lambda)}(1)=1$ for every $r\ge-1$.

\begin{lemma}\label{lem:representation}
Fix a real number $x$.  For $m\ge k\ge1$, put
\begin{equation}\label{eq:representation-coefficients}\begin{aligned}
a_k&=\frac{(k+\alpha+1)(k+2\alpha+1)}
{(\alpha+1)(2k+2\alpha+1)},\\[7pt]
b_k&=\frac{k(k+\alpha)}
{(\alpha+1)(2k+2\alpha+1)}.
\end{aligned}\end{equation}
Then $a_k,b_k>0$ and
\begin{equation}\label{eq:representation}
U_m(x)=a_kW_{m-k}^{(k,\lambda)}(x)Q_k(x)
-b_kW_{m-k-1}^{(k+1,\lambda)}(x)Q_{k-1}(x).
\end{equation}
\end{lemma}

\begin{proof}
At $m=k$, \eqref{eq:representation} is exactly
\eqref{eq:contiguous-representation}. Also, \eqref{eq:W-recurrence} gives
\begin{equation}\label{eq:W-first-difference}
W_1^{(k,\lambda)}(x)-1
=-\frac{2k+2\alpha+1}{k+2\alpha+1}(1-x).
\end{equation}
Combining \eqref{eq:W-first-difference} with
\eqref{eq:contiguous-difference} proves \eqref{eq:representation}
at $m=k+1$. Both sides then satisfy
\eqref{eq:ultraspherical-recurrence} in $m$, so the identity holds for every
$m\ge k$.
\end{proof}

If $x\in[\xi_{k-1},\xi_k]$, interlacing
\cite[Theorem~3.3.2]{Szego1975} gives
\begin{equation}\label{eq:interval-signs}
Q_j(x)\ge0,\quad 0\le j\le k-1,\quad Q_k(x)\le0.
\end{equation}
Indeed, $x$ is to the right of the largest zero of $Q_j$ for $j<k$.
For $k\ge2$, the number $\xi_{k-1}$ lies between the two largest zeros
of $Q_k$; for $k=1$, the sign of $Q_1$ on $[-1,\xi_1]$ is immediate.
Relation \eqref{eq:contiguous-difference} and \eqref{eq:interval-signs} already give
$U_0(x)\ge\cdots\ge U_k(x)\le U_{k+1}(x)$.
Subtracting the case $m=k$ from \eqref{eq:representation} gives
\begin{equation}\label{eq:representation-difference}\begin{aligned}
U_m(x)-U_k(x)={}&a_k\bigl(W_{m-k}^{(k,\lambda)}(x)-1\bigr)Q_k(x)\\[7pt]
&-b_k\bigl(W_{m-k-1}^{(k+1,\lambda)}(x)-1\bigr)Q_{k-1}(x).
\end{aligned}\end{equation}
By \eqref{eq:interval-signs} and \eqref{eq:representation-difference},
it is enough to prove
$W_{m-k}^{(k,\lambda)}(x)\le1$ and
$W_{m-k-1}^{(k+1,\lambda)}(x)\le1$ to obtain $U_m(x)\ge U_k(x)$.

\subsection{Positive connection coefficients}

\begin{proposition}\label{prop:positive-connection}
Let $\lambda\ge1/2$, $c>0$, and $K=c/(2\lambda)$. For every $N\ge0$,
there are numbers $d_{N,j}\ge0$, $0\le j\le N$, such that
\begin{equation}\label{eq:positive-connection}
W_N^{(c,\lambda)}(x)
=\sum_{j=0}^N d_{N,j}W_j^{(K,1/2)}(x),
\quad
\sum_{j=0}^N d_{N,j}=1.
\end{equation}
\end{proposition}

\begin{proof}
Put $h=c/\lambda$, so $K=h/2$. Let $p_n$ and $q_n$ be the monic
normalisations of $W_n^{(c,\lambda)}$ and $W_n^{(K,1/2)}$, respectively.
The two $W$ families have positive leading coefficients by
\eqref{eq:W-recurrence}.
That recurrence gives $p_0=q_0=1$, $p_1=q_1=x-\beta_0$, and
\begin{equation}\label{eq:monic-recurrence}
\begin{aligned}
xp_n&=p_{n+1}+\beta_np_n+\eta_n(\lambda)p_{n-1},\\[7pt]
xq_n&=q_{n+1}+\beta_nq_n+\eta_n(1/2)q_{n-1},
\end{aligned}
\end{equation}
for $n\ge1$, where
\begin{equation}\label{eq:monic-diagonal}
\beta_0=\frac{h}{2(h+1)},\quad \beta_n=0\quad(n\ge1).
\end{equation}
For $a\ge1/2$ and $n\ge1$, write $t_n(a)=a(h+1)+n$. The lower
coefficient is
\begin{equation}\label{eq:monic-coefficients}
\begin{aligned}
\eta_n(a)
&=\frac{(ah+n)(a(h+2)+n-1)}
       {4(a(h+1)+n-1)(a(h+1)+n)}\\[7pt]
&=\frac14-\frac{a(a-1)}{4(t_n(a)-1)t_n(a)}>0.
\end{aligned}
\end{equation}

The sequence $\eta_n(1/2)$ decreases with $n$ and is larger than $1/4$.
For $1/2\le a\le1$, the inequalities $a(1-a)\le1/4$ and
$t_n(a)\ge t_n(1/2)>1$ give $\eta_n(a)\le\eta_n(1/2)$.
For $a\ge1$, this follows from $\eta_n(a)\le1/4$. Hence
\begin{equation}\label{eq:monic-coefficient-order}
\eta_m(1/2)\ge\eta_n(\lambda)>0\quad(1\le m\le n),
\quad
\beta_j\ge\beta_n\quad(0\le j\le n).
\end{equation}

Expand in the monic basis:
\begin{equation}\label{eq:monic-expansion}
p_n=\sum_{j=0}^n c_{n,j}q_j.
\end{equation}
Then $c_{n,n}=1$. Set $c_{n,j}=0$ for $j<0$ or $j>n$.
We prove by induction that
\begin{equation}\label{eq:connection-induction}
c_{n,j}\ge c_{n-1,j-1}\ge0,\quad 0\le j\le n,\quad n\ge1.
\end{equation}
For $n=1$, this follows from $p_1=q_1$:
$c_{1,0}=0$ and $c_{1,1}=c_{0,0}=1$.
Substituting \eqref{eq:monic-expansion} into \eqref{eq:monic-recurrence}
and comparing coefficients gives, for $n\ge1$ and $0\le j\le n+1$,
\begin{equation}\label{eq:connection-recurrence}
\begin{aligned}
c_{n+1,j}-c_{n,j-1}
={}&(\beta_j-\beta_n)c_{n,j}\\[7pt]
&+\eta_{j+1}(1/2)c_{n,j+1}-\eta_n(\lambda)c_{n-1,j}.
\end{aligned}
\end{equation}
For $0\le j<n$, the right-hand side can be written as
\begin{equation}\label{eq:connection-nonnegative-step}
\begin{aligned}
&(\beta_j-\beta_n)c_{n,j}\\[7pt]
&\quad+\bigl(\eta_{j+1}(1/2)-\eta_n(\lambda)\bigr)c_{n,j+1}\\[7pt]
&\quad+\eta_n(\lambda)\bigl(c_{n,j+1}-c_{n-1,j}\bigr).
\end{aligned}
\end{equation}
Each term is non-negative by \eqref{eq:monic-coefficient-order} and the
induction hypothesis. At $j=n$, \eqref{eq:connection-recurrence} gives
$c_{n+1,n}=c_{n,n-1}$; at $j=n+1$, it gives
$c_{n+1,n+1}=c_{n,n}=1$. This proves \eqref{eq:connection-induction},
so all $c_{n,j}$ are non-negative, including $c_{0,0}=1$.

Let $L_n,M_n>0$ be the leading coefficients of $W_n^{(c,\lambda)}$ and
$W_n^{(K,1/2)}$, respectively. Multiplying \eqref{eq:monic-expansion}
by $L_N$, with $n=N$, gives the expansion in \eqref{eq:positive-connection}
with
\begin{equation}\label{eq:convex-normalisation}
d_{N,j}=\frac{L_N}{M_j}c_{N,j}\ge0.
\end{equation}
The initial values and \eqref{eq:W-recurrence} give
$W_n^{(c,\lambda)}(1)=W_n^{(K,1/2)}(1)=1$ for every $n\ge0$.
Evaluation at $x=1$ therefore gives $\sum_{j=0}^N d_{N,j}=1$.
\end{proof}

\section{An upper bound for the auxiliary polynomials}\label{sec:auxiliary-bound}

We first prove the bound at $\lambda=1/2$.
The proof combines a rotation estimate for the initial degrees with
decay of a quadratic energy for the later degrees.
Proposition~\ref{prop:positive-connection} then transfers the result
to every $\lambda\ge1/2$.

\begin{lemma}\label{thm:Legendre-plateau}
Let $K>0$, $x=\cos\theta$, and
\begin{equation}\label{eq:auxiliary-angle-range}
0\le\theta\le\min\left\{\frac{\pi}{2},\frac{2\pi}{K}\right\}.
\end{equation}
Then
\begin{equation}\label{eq:legendre-upper-bound}
W_N^{(K,1/2)}(x)\le1\quad\text{for every}\quad N\ge-1.
\end{equation}
\end{lemma}

\begin{proof}
At the fixed point $x=\cos\theta$, write
$q_n=W_n^{(K,1/2)}(x)$ for $n\ge-1$.
Then $q_{-1}=q_0=1$, and \eqref{eq:W-recurrence} becomes
\begin{equation}\label{eq:legendre-W-recurrence}
(K+n+1)q_{n+1}
=(2K+2n+1)xq_n-(K+n)q_{n-1},\quad n\ge0.
\end{equation}
The cases $N=-1$ and $x=1$ are immediate.  Hence assume $N\ge0$ and
$0\le x<1$.
\smallskip
\noindent\emph{The rotation estimate.}
Put $\sigma=\sin\theta$. For $n\ge0$, set
\begin{equation}\label{eq:phase-vector}
z_n=xq_n-q_{n-1},\quad
v_n=\begin{pmatrix}\sigma q_n\\[7pt]z_n\end{pmatrix}.
\end{equation}
For $s>0$, put $\rho_s=s/(s+1)$ and
$D_s=\mathop{\rm diag}(1,\rho_s)$.  Then
\eqref{eq:legendre-W-recurrence} becomes
\begin{equation}\label{eq:rotation-contraction}
v_{n+1}=RD_{K+n}v_n,
\quad
R=\begin{pmatrix}x&\sigma\\[7pt]-\sigma&x\end{pmatrix},\quad n\ge0.
\end{equation}
Also
\begin{equation}\label{eq:initial-vector}
v_0=2\sin(\theta/2)
\begin{pmatrix}\cos(\theta/2)\\[7pt]-\sin(\theta/2)\end{pmatrix}.
\end{equation}
Let $I$ be the identity and $P$ the projection onto the first
coordinate. Since $D_s=\rho_sI+(1-\rho_s)P$,
iteration of \eqref{eq:rotation-contraction} expresses $v_N$ as a
convex combination of the products obtained by choosing $I$ or $P$ at
each step.  If the projections occur at the steps
$0\le i_1<\cdots<i_r\le N-1$, its first coordinate is
\begin{equation}\label{eq:projection-product}
2\sin(\theta/2)\cos((i_1+1/2)\theta)
\prod_{j=2}^r\cos((i_j-i_{j-1})\theta)
\cos((N-i_r)\theta).
\end{equation}
If there is no projection, the first coordinate is
$2\sin(\theta/2)\cos((N+1/2)\theta)$.  Thus in every case it has the form
\begin{equation}\label{eq:cosine-product}
2\sin(\theta/2)\prod_{j=0}^r\cos(a_j\theta),
\end{equation}
where one of $a_0,\ldots,a_r$ is a positive half-integer, the others
are positive integers, and $\sum_{j=0}^r a_j=N+1/2$.
Expanding the product in \eqref{eq:cosine-product} gives an average of cosines
$\cos(b\theta)$, where $b$ is a nonzero half-integer and
$|b|\le N+1/2$.  If $(N+1)\theta\le2\pi$, then
\begin{equation}\label{eq:rotation-frequencies}
\frac{\theta}{2}\le |b|\theta\le2\pi-\frac{\theta}{2},
\end{equation}
so every such cosine is at most $\cos(\theta/2)$.  Hence the first
coordinate of $v_N$ is at most $\sigma$, and
\begin{equation}\label{eq:rotation-bound}
q_N\le1\quad\text{whenever}\quad (N+1)\theta\le2\pi.
\end{equation}

\smallskip
\noindent\emph{A two-step energy estimate.}
To control the later degrees, put $y=1-x$ and define the energy
\begin{equation}\label{eq:energy-definition}
E_n=\lVert v_n\rVert^2
=\sigma^2q_n^2+z_n^2.
\end{equation}
For $s>0$, put
$\delta_s=1-\rho_s^2=(2s+1)/(s+1)^2$.
Since $R$ is orthogonal, \eqref{eq:rotation-contraction} gives
\begin{equation}\label{eq:energy-dissipation}
E_{n+1}=E_n-\delta_{K+n}z_n^2,\quad n\ge0.
\end{equation}
Relation \eqref{eq:rotation-contraction} also gives
\begin{equation}\label{eq:weak-two-step-form}
z_n^2+z_{n+1}^2\ge yE_n.
\end{equation}
Indeed, if $u=\sigma q_n$ and $v=z_n$, then
\begin{equation}\label{eq:weak-quadratic-identity}
z_n^2+z_{n+1}^2-yE_n
=x\left[y\left(u-\frac{\sigma\rho_{K+n}}{y}v\right)^2
+\delta_{K+n}v^2\right]\ge0.
\end{equation}
Since $\delta_s$ decreases with $s$, combining
\eqref{eq:energy-dissipation} at two consecutive steps with
\eqref{eq:weak-two-step-form} gives
\begin{equation}\label{eq:weak-energy-step}
E_{n+2}\le
(1-y\delta_{K+n+1})E_n,\quad n\ge0.
\end{equation}
The initial values are
$E_0=2y$ and
$E_1=2y(1-y\delta_K/2)$.  Thus, for
$N=2M,\quad M\ge0$,
\begin{equation}\label{eq:weak-even-product}
E_N\le2y\prod_{j=0}^{M-1}
(1-y\delta_{K+2j+1}),
\end{equation}
whereas, for $N=2M+1,\quad M\ge0$,
\begin{equation}\label{eq:weak-odd-product}
E_N\le2y(1-y\delta_K/2)
\prod_{j=0}^{M-1}(1-y\delta_{K+2j+2}).
\end{equation}
Each product in \eqref{eq:weak-even-product} and
\eqref{eq:weak-odd-product} has the form
$F(y)=\prod_{i=1}^{\ell}(1-a_i y)$, where
$0\le a_i\le1,\quad 1\le i\le\ell$.  It is convex on $[0,1]$, since all
its factors are non-negative and direct differentiation gives
$F''(y)\ge0$.
Moreover,
\begin{equation}\label{eq:telescoping-factor}
1-\delta_s=\left(\frac{s}{s+1}\right)^2\le\frac{s}{s+2}.
\end{equation}
Using \eqref{eq:telescoping-factor} in \eqref{eq:weak-even-product},
the product at $y=1$ is at most
\begin{equation}\label{eq:weak-even-endpoint}
\frac{K+1}{K+N+1}.
\end{equation}
For \eqref{eq:weak-odd-product}, it is at most
\begin{equation}\label{eq:weak-odd-endpoint}
\left(1-\frac{\delta_K}{2}\right)
\frac{K+2}{K+N+1}.
\end{equation}
Both \eqref{eq:weak-even-endpoint} and \eqref{eq:weak-odd-endpoint}
are at most $1/2$ when $N\ge K+1$; for the second one
we also use $\delta_K\ge2/(K+2)$.  Convexity gives
$F(y)\le1-y/2$, and hence
\begin{equation}\label{eq:weak-energy-bound}
E_N\le2y-y^2=\sigma^2.
\end{equation}
By \eqref{eq:energy-definition} and \eqref{eq:energy-dissipation},
this implies $|q_n|\le1$ for every $n\ge N$.

\smallskip
\noindent\emph{The sharper bound for $K\ge6$.}
For $x\ge1/2$, we strengthen \eqref{eq:weak-two-step-form} to
\begin{equation}\label{eq:strong-two-step-form}
z_n^2+z_{n+1}^2\ge c_syE_n,
\quad c_s=\frac{2(s+2)}{2s+3},\quad s=K+n>0.
\end{equation}
The quadratic form obtained by subtracting the right-hand side
of \eqref{eq:strong-two-step-form} from the left has first principal minor
$y(x-1/(2s+3))>0$, and its determinant divided by $y$ is
\begin{equation}\label{eq:strong-determinant}
(c_s-1)^2+(1-c_s^2)x+c_s\delta_sx^2.
\end{equation}
The polynomial in \eqref{eq:strong-determinant} is convex.
At $x=1/2$, its value and derivative are
\begin{equation}\label{eq:strong-determinant-endpoint}
\frac{3s^2+5s+1}{2(s+1)^2(2s+3)^2}>0,
\quad
\frac{4s^3+17s^2+20s+5}{(s+1)^2(2s+3)^2}>0.
\end{equation}
The quadratic form is therefore positive definite by the principal-minor
criterion \cite[Theorem~7.2.5(b), p.~439]{HornJohnson2013}.
Since $c_s\delta_{s+1}=2/(s+2)$, combining
\eqref{eq:strong-two-step-form} with \eqref{eq:energy-dissipation} gives
\begin{equation}\label{eq:strong-energy-step}
E_{n+2}\le
\left(1-\frac{2y}{K+n+2}\right)E_n,\quad n\ge0.
\end{equation}

Suppose that $K\ge6$ and set $N=\lfloor K\rfloor$.  Then
$x\ge\cos(2\pi/K)\ge1/2$ and $N\ge6$.
Iteration of \eqref{eq:strong-energy-step}, starting with
$E_0=2y$ or $E_1=2y(1-y\delta_K/2)$, gives
$E_N\le2yF_N(y)$, where
\begin{equation}\label{eq:strong-products}
\begin{aligned}
F_{2M}(y)&=\prod_{j=0}^{M-1}\left(1-\frac{2y}{K+2j+2}\right),\\[7pt]
F_{2M+1}(y)&=\left(1-\frac{y\delta_K}{2}\right)
\prod_{j=0}^{M-1}\left(1-\frac{2y}{K+2j+3}\right).
\end{aligned}
\end{equation}
If $N=2M$, then $K<2M+1$, and \eqref{eq:strong-products} gives
\begin{equation}\label{eq:even-comparison-product}
F_N(1/2)<C_M,
\quad
C_M=\prod_{j=0}^{M-1}
\frac{2M+2j+2}{2M+2j+3}.
\end{equation}
For $M\ge3$, the numbers $C_M$ decrease because
\begin{equation}\label{eq:comparison-product-ratio}
\frac{C_{M+1}}{C_M}
=\frac{(4M+2)(4M+4)(2M+3)}
{(2M+2)(4M+3)(4M+5)}<1;
\end{equation}
the difference between the denominator and numerator is $6M+6$.
Hence
\begin{equation}\label{eq:comparison-product-bound}
C_M\le C_3=\frac{320}{429}<\frac34.
\end{equation}
If $N=2M+1$, use $K<2M+2$ in \eqref{eq:strong-products}.
Direct cancellation of the factors gives
\begin{equation}\label{eq:odd-comparison-product}
F_N(1/2)<C_M\left(
1-\frac{2M+5}{4(M+1)(2M+3)(4M+3)}\right)<C_M.
\end{equation}
Combining \eqref{eq:even-comparison-product},
\eqref{eq:comparison-product-bound}, and \eqref{eq:odd-comparison-product}
gives $F_N(1/2)<3/4$ in both cases.  Convexity on $[0,1/2]$ now gives
$F_N(y)\le1-y/2$.  Therefore $E_N\le\sigma^2$ and
$|q_n|\le1$ for $n\ge N$.  Estimate \eqref{eq:rotation-bound} covers $n<N$, since
$(n+1)\theta\le N\theta\le K\theta\le2\pi$.

\smallskip
\noindent\emph{The remaining degrees for $K<6$.}
The rotation bound
\eqref{eq:rotation-bound} and the weak energy bound
\eqref{eq:weak-energy-bound}, valid for $N\ge K+1$, leave only
degree $4$ for $3<K<5$, degree $5$ for $4<K<6$, and degree $6$
for $5<K<6$. Lemma~\ref{lem:exceptional-auxiliary-degrees} in
Appendix~\ref{app:exceptional} covers degree $4$ on $[0,1]$
and degree $5$ when $x\le1/2$. If $x\ge1/2$, then $6\theta\le2\pi$, so
\eqref{eq:rotation-bound} also covers degree $5$.
It therefore remains to control degree $6$ for $5<K<6$.
In this range,
$x\ge\cos(2\pi/5)>3/10$, so $y<7/10$.  Since $\delta_s$ decreases,
\begin{equation}\label{eq:degree-six-deltas}
\delta_{K+1}\ge\delta_7=\frac{15}{64},\quad
\delta_{K+3}\ge\delta_9=\frac{19}{100},\quad
\delta_{K+5}\ge\delta_{11}=\frac{23}{144}.
\end{equation}
Using \eqref{eq:degree-six-deltas} in \eqref{eq:weak-even-product} with $N=6$ gives
\begin{equation}\label{eq:degree-six-product}
E_6\le2yF_*(y),
\quad
F_*(y)=\left(1-\frac{15y}{64}\right)
\left(1-\frac{19y}{100}\right)
\left(1-\frac{23y}{144}\right).
\end{equation}
The product $F_*$ is convex, $F_*(0)=1$, and
\begin{equation}\label{eq:degree-six-product-bound}
F_*(7/10)<\frac{21}{25}\frac{87}{100}\frac89
=\frac{406}{625}<\frac{13}{20}.
\end{equation}
Hence $F_*(y)\le1-y/2$, and $|q_n|\le1$ for every $n\ge6$.

\end{proof}

\begin{corollary}\label{cor:auxiliary-upper-bound}
Let $\lambda\ge1/2$, $c>0$, and
$0\le\theta\le\min\{\pi/2,4\pi\lambda/c\}$. Then
\begin{equation}\label{eq:auxiliary-upper-bound}
W_N^{(c,\lambda)}(\cos\theta)\le1\quad\text{for every}\quad N\ge-1.
\end{equation}
\end{corollary}

\begin{proof}
The case $N=-1$ follows from \eqref{eq:W-recurrence}.
For $N\ge0$, Proposition~\ref{prop:positive-connection} expresses
$W_N^{(c,\lambda)}$ as a convex combination of
$W_j^{(K,1/2)}$, where $K=c/(2\lambda)$.
The angular hypothesis is precisely $\theta\le2\pi/K$, so
Lemma~\ref{thm:Legendre-plateau} bounds each term by one.
\end{proof}

\section{The largest zero and the positive intervals}\label{sec:largest-zero}

\begin{proposition}\label{prop:angle}
Let $k\ge3$ and put $\theta_*=\arccos\xi_{k-1}$. Then
\begin{equation}\label{eq:largest-zero-angle}
\theta_*\le\frac{4\pi\lambda}{k+1}.
\end{equation}
\end{proposition}

\begin{proof}
Put $r=k-1$ and $d=2\lambda+1\ge2$.  Let $A$ be the tridiagonal matrix
with rows and columns indexed by $0,\ldots,r$.  Its upper and lower
entries are
\begin{equation}\label{eq:spectral-matrix-entries}\begin{aligned}
a_j&=\frac{j+2\lambda}{2(j+\lambda)},\quad 0\le j<r,\\[7pt]
c_j&=\frac{j}{2(j+\lambda)},\quad 1\le j\le r.
\end{aligned}\end{equation}
The only nonzero diagonal entry is
$A(r,r)=a_r$, where $a_r$ is given by the same formula.  The last row
has the following meaning: for a real number $t$, the vector
$(U_0(t),\ldots,U_r(t))^{\mathsf T}$
is an eigenvector with eigenvalue $t$ precisely when
$U_{r+1}(t)=U_r(t)$.  Recurrence \eqref{eq:ultraspherical-recurrence}, followed by
\eqref{eq:contiguous-difference}, shows that the eigenvalues of $A$ are $1$ and the zeros of
$Q_r$.
Since $Q_2(0)=-1/(d+4)<0<Q_2(1)$, interlacing gives
$\xi_r>0$ \cite[Theorem~3.3.2]{Szego1975}.  In particular,
$\theta_*<\pi/2$.

Set $w_{-1}=0$ and introduce the weights
\begin{equation}\label{eq:spectral-weights}
w_j=\frac{(d)_j}{j!},\quad b_j=w_j+w_{j-1},\quad j\ge0.
\end{equation}
For any real number $\gamma$, put $(\gamma)_0=1$ and
$(\gamma)_j=\gamma(\gamma+1)\cdots(\gamma+j-1),\quad j\ge1$.
We have
\begin{equation}\label{eq:detailed-balance}
b_ja_j=w_j=b_{j+1}c_{j+1},\quad 0\le j<r.
\end{equation}
Let $B=\mathop{\rm diag}(b_0,\ldots,b_r)$. Identities \eqref{eq:detailed-balance} give
$BA=A^{\mathsf T}B$, so $B^{1/2}AB^{-1/2}$ is real symmetric and
similar to $A$.
Similarity and the spectral theorem justify the $b$-weighted spectral
decomposition used below
\cite[Sections~1.3 and~4.1]{HornJohnson2013}.
The number $\xi_r$ is the second largest eigenvalue.
For every non-constant vector
$f=(f_0,\ldots,f_r)^{\mathsf T}$,
the Rayleigh theorem in the $b$-weighted inner product
\cite[Section~4.2]{HornJohnson2013}, applied after subtracting the
weighted mean, gives
\begin{equation}\label{eq:rayleigh-quotient}
1-\xi_r\le
\frac{\sum_{j=0}^{r-1}w_j(f_{j+1}-f_j)^2}
{\sum_{j=0}^rb_j(f_j-\overline f)^2},
\quad
\overline f=\frac{\sum_{j=0}^rb_jf_j}{\sum_{j=0}^rb_j}.
\end{equation}
Take $f_j=j$.  The identity
\begin{equation}\label{eq:pochhammer-sum}
\sum_{j=0}^m\frac{(\gamma)_j}{j!}
=\frac{(\gamma+1)_m}{m!},\quad m\ge0,
\end{equation}
and its first two moment versions make the calculation explicit.  If
\begin{equation}\label{eq:moment-factor}
C=\frac{(d+1)_{r-1}}{(r-1)!},
\end{equation}
then
\begin{equation}\label{eq:spectral-moments}\begin{aligned}
\sum_{j=0}^r b_j
&=C\frac{d+2r}{r},\\[7pt]
\sum_{j=0}^r jb_j
&=C\frac{d^2+2d r+1}{d+1},\\[7pt]
\sum_{j=0}^r j^2b_j
&=C\left(
\frac{d(r-1)(d+2r-2)}{d+2}
+\frac{d(d+4r-3)}{d+1}+1\right).
\end{aligned}\end{equation}
The numerator in \eqref{eq:rayleigh-quotient} is $C$.
Substituting \eqref{eq:spectral-moments} into that quotient gives
\begin{equation}\label{eq:rational-zero-bound}
1-\xi_r\le\frac{(d+1)^2(d+2)(d+2r)}
{(d+r)(d^3+4d^2r+4d r^2+d+2)}.
\end{equation}

Set
\begin{equation}\label{eq:angle-parameter}
\tau=\frac{d-1}{r+2}=\frac{2\lambda}{k+1}.
\end{equation}
Dividing \eqref{eq:rational-zero-bound} by $\tau^2$ and using
$d\ge2$ and $r\ge2$ gives
\begin{equation}\label{eq:rational-bound-estimate}
\frac{1-\xi_r}{\tau^2}
\le
\left(\frac{d+1}{d-1}\right)^2\frac{d+2}{d}
\frac{(r+2)^2}{(d+r)(d+2r)}
\le12.
\end{equation}
The three factors in \eqref{eq:rational-bound-estimate} are at most $9$, $2$, and $2/3$;
for the last one use $d+r\ge r+2$ and
$d+2r\ge2r+2$.
If $\tau\le1/4$, concavity of sine gives
\begin{equation}\label{eq:sine-chord}
\sin(\pi\tau)\ge2\sqrt2\tau.
\end{equation}
Combining \eqref{eq:rational-bound-estimate} and \eqref{eq:sine-chord} yields
\begin{equation}\label{eq:angle-comparison}
1-\cos(2\pi\tau)=2\sin^2(\pi\tau)
\ge16\tau^2\ge12\tau^2\ge1-\xi_r.
\end{equation}
The monotonicity of $1-\cos\theta$ on $[0,\pi]$ gives
$\theta_*\le2\pi\tau$.  If $\tau\ge1/4$, the same inequality is
automatic because $\theta_*<\pi/2\le2\pi\tau$.  Thus
$\theta_*\le2\pi\tau=4\pi\lambda/(k+1)$ in every case.
\end{proof}

\subsection{The intervals with \texorpdfstring{$k\ge3$}{k >= 3}}

We have already seen that $\xi_2>0$.  Interlacing
\cite[Theorem~3.3.2]{Szego1975} gives $\xi_{k-1}>0$ for $k\ge3$.  Thus
$x\in[\xi_{k-1},\xi_k]$ can be written as $x=\cos\theta$ with
$0\le\theta<\pi/2$.  Since $x\ge\xi_{k-1}$,
Proposition~\ref{prop:angle} gives
\begin{equation}\label{eq:application-angle}
\theta\le\theta_*\le\frac{4\pi\lambda}{k+1}
\le\frac{4\pi\lambda}{k}.
\end{equation}
Corollary~\ref{cor:auxiliary-upper-bound} applies to both $c=k$ and $c=k+1$.  Hence
the two values $W_{m-k}^{(k,\lambda)}(x)$ and
$W_{m-k-1}^{(k+1,\lambda)}(x)$ in Lemma~\ref{lem:representation} are
at most one. Substituting these bounds and \eqref{eq:interval-signs}
into \eqref{eq:representation-difference} gives
$U_m(x)\ge U_k(x),\quad m>k$.  The preceding degrees were already
controlled by \eqref{eq:contiguous-difference} and \eqref{eq:interval-signs}.  This proves
Theorem~\ref{thm:main} when $k\ge3$.

\section{The first two intervals and completion}\label{sec:first-intervals}

It remains to treat $k=1,2$, where the intervals contain negative values
of $x$. Recall that $d=2\alpha+2=2\lambda+1\ge2$.
We use the normalised integral representation
\begin{equation}\label{eq:integral-representation}
U_n(x)=\mathbb E\bigl(x+i\sqrt{1-x^2}S\bigr)^n,\quad
\mathbb ES^{2j}=\frac{(2j-1)!!}{d(d+2)\cdots(d+2j-2)},\quad j\ge0,
\end{equation}
where $S$ has density proportional to $(1-s^2)^{\lambda-1}$ on
$[-1,1]$, and $\mathbb E$ denotes expectation with respect to $S$.
Here $(-1)!!=1$.  The symmetry of $S$ makes the expectation
real.  The representation follows, after normalisation, from
\cite[(18.10.4)]{DLMF}; the moments follow from the beta integral
\cite[(5.12.1)]{DLMF}.

Put
\begin{equation}\label{eq:initial-thresholds}
t_0=\frac1{d+1},\quad t_1=\frac1{1+\sqrt{d+4}}.
\end{equation}
The first two breakpoints and the second polynomial are
\begin{equation}\label{eq:initial-endpoints}
\xi_1=-t_0,\quad \xi_2=t_1,\quad
U_2(x)=\frac{(d+1)x^2-1}{d}.
\end{equation}
In particular, $U_2(x)<0$ for $-t_0\le x\le t_1$.

\begin{lemma}\label{lem:large-initial-degrees}
For every $n\ge7$,
\begin{equation}\label{eq:large-initial-degrees}\begin{aligned}
U_n(x)&\ge x,\quad -1\le x\le-t_0,\\[7pt]
U_n(x)&\ge U_2(x),\quad -t_0\le x\le t_1.
\end{aligned}\end{equation}
\end{lemma}

\begin{proof}
For $0\le t\le1$, put
\begin{equation}\label{eq:moment-majorant}
Y_t=t^2+(1-t^2)S^2,
\quad
M(t)=\mathbb E\frac{Y_t^3+Y_t^4}{2}.
\end{equation}
By \eqref{eq:integral-representation}, if $|x|=t$, then
\begin{equation}\label{eq:moment-bound}
U_n(x)\ge-\mathbb EY_t^{n/2}\ge-M(t),\quad n\ge7.
\end{equation}
The second inequality follows from
$2Y_t^{7/2}\le Y_t^3+Y_t^4$.

The function $M$ is convex: both $Y_t$ and the function
$u\mapsto(u^3+u^4)/2$ are convex and non-decreasing on $[0,1]$.
Thus $F(t)=t-M(t)$ is concave.  Put $a=d-2$.
The moments in \eqref{eq:integral-representation} give
\begin{equation}\label{eq:concavity-endpoints}
F(1)=0,\quad
F(t_0)=\frac{p(a)}{2(a+3)^7(a+6)(a+8)},
\end{equation}
where
\begin{equation}\label{eq:concavity-polynomial}\begin{aligned}
p(a)={}&2a^8+64a^7+855a^6+6129a^5+25545a^4\\[7pt]
&+62399a^3+83886a^2+50144a+3968.
\end{aligned}\end{equation}
All coefficients in \eqref{eq:concavity-polynomial} are positive.
Since $a\ge0$, \eqref{eq:concavity-endpoints} and concavity give $M(t)\le t$ for
$t_0\le t\le1$.  Taking $x=-t$ proves the
first assertion.

For the second assertion, define
\begin{equation}\label{eq:quadratic-majorant}
G(y)=\frac{1-(d+1)y}{d}-M(\sqrt y),\quad 0\le y\le1.
\end{equation}
The function $G$ decreases: its first term decreases, while
$M(\sqrt y)$ increases. Since
$(1-(d+1)t_0^2)/d=t_0$, \eqref{eq:concavity-endpoints} gives
$G(t_0^2)=F(t_0)>0$.  This is enough when
$t_1\le t_0$.  Suppose that $t_1\ge t_0$ and put
$\rho=\sqrt{d+4}$.  Then $\rho^2-\rho-4\ge0$ and, since
$0\le Y_{t_1}\le1$,
$M(t_1)\le\mathbb E(Y_{t_1}^3)$.
Using \eqref{eq:initial-endpoints} and the moments in
\eqref{eq:integral-representation} gives
\begin{equation}\label{eq:second-moment-endpoint}\begin{aligned}
-U_2(t_1)-\mathbb E(Y_{t_1}^3)
={}&\frac{2}{(\rho-2)(\rho+1)^5(\rho^2-2)}\\[7pt]
&\times\bigl((\rho^2-\rho-4)
(\rho^3+4\rho^2+9\rho+6)+12\rho+20\bigr)>0.
\end{aligned}\end{equation}
Thus $G(t_1^2)>0$ also in this case.  Since $G$ decreases,
$G(y)\ge0$ for
$0\le y\le\max\{t_0^2,t_1^2\}$.  If $-t_0\le x\le t_1$, then
\begin{equation}\label{eq:large-second-interval}
U_n(x)\ge-M(|x|)\ge U_2(x).
\end{equation}
\end{proof}

\begin{proof}[Completion of the proof of Theorem~\ref{thm:main}]
On the first interval $[-1,\xi_1]=[-1,-t_0]$,
Lemma~\ref{lem:large-initial-degrees} proves $U_n\ge U_1$ for $n\ge7$,
and Lemma~\ref{lem:first-cell} in Appendix~\ref{app:initial-degrees}
proves the same inequality for $0\le n\le6$.
On the second interval $[\xi_1,\xi_2]=[-t_0,t_1]$, the other inequality
in Lemma~\ref{lem:large-initial-degrees} and Lemma~\ref{lem:second-cell}
give $U_n\ge U_2$ for every $n\ge0$.
Section~\ref{sec:largest-zero} established the result for $k\ge3$.
This completes the proof.

The endpoint ties follow from \eqref{eq:contiguous-difference}:
$U_{k-1}(\xi_{k-1})=U_k(\xi_{k-1})$ for $k\ge2$, and
$U_k(\xi_k)=U_{k+1}(\xi_k)$. At the remaining endpoint,
$U_n(-1)=(-1)^n$.
\end{proof}

\section{Minimising indices}\label{sec:minimising-indices}

We retain the notation $U_n=R_n^{(\alpha,\alpha)}$ and
$Q_n=R_n^{(\alpha+1,\alpha)}$.
Theorem~\ref{thm:main} gives the following consequence for the least
degree attaining the minimum.

\begin{corollary}\label{cor:least-minimiser}
Let $\alpha\ge0$ and $-1<x<1$. If $k$ is the least index for which
$U_k(x)=\min_{n\ge0}U_n(x)$, then
\begin{equation}\label{eq:least-minimiser-descent}
U_0(x)\ge U_1(x)
\ge\cdots\ge U_k(x).
\end{equation}
\end{corollary}

\begin{proof}
By \eqref{eq:envelope-breakpoints}, choose $j\ge1$ with
$x\in[\xi_{j-1},\xi_j]$.
Theorem~\ref{thm:main} shows that the minimum is attained at degree $j$,
so $k$ exists and $k\le j$.
The signs in \eqref{eq:interval-signs} and the contiguous relation
\eqref{eq:contiguous-difference} give
$U_0(x)\ge U_1(x)\ge\cdots\ge U_j(x)$, and hence
\eqref{eq:least-minimiser-descent}.
\end{proof}

The qualification \emph{least} matters when several degrees attain the
same minimum. In Question~1 of his thesis, de Oliveira Filho asked the
following \cite[Section~3.5d, p.~47]{OliveiraFilho2009}; Koornwinder
reproduced it in \cite[Topic~11, p.~21]{Koornwinder2010}.
\begin{question}[de Oliveira Filho's Question 1]\label{prob:first}
Let $\alpha\ge0$, $-1<x<1$, and let $k\ge0$ be an integer satisfying
\begin{equation}\label{eq:arbitrary-minimiser}
U_k(x)=\min_{n\ge0}U_n(x).
\end{equation}
Must $(U_n(x))_{n=0}^k$ be decreasing?
\end{question}

After replacing de Oliveira Filho's notation by $U$, this is Question~1 of
the thesis; Koornwinder added the parenthetical ``such $k$ exists''.
Since $U_1(x)=x<1$, every minimising index is positive.
We interpret $k$ as an arbitrary index satisfying
\eqref{eq:arbitrary-minimiser}, and ``decreasing'' as non-increasing. The strict rise below makes
the latter convention immaterial.

In his Remark~1, Koornwinder used \eqref{eq:contiguous-difference} to claim that every minimising degree $k$ must satisfy
\begin{equation}\label{eq:minimiser-interval}
\xi_{k-1}\le x\le\xi_k.
\end{equation}
Under \eqref{eq:minimiser-interval}, interlacing and
\eqref{eq:contiguous-difference} show that
$(U_n(x))_{n=0}^k$ is non-increasing. This
interval condition is not necessary for an arbitrary minimising degree,
as the Legendre case $\alpha=0$ shows.
In this case, write $P_n=U_n=P_n^{(0,0)}$.
At $x=-1/3$, direct
evaluation gives
\begin{equation}\label{eq:legendre-counterexample}\begin{aligned}
P_0(-1/3)&=1,\quad
P_1(-1/3)=P_2(-1/3)=P_5(-1/3)=-\frac13,\\[7pt]
P_3(-1/3)&=\frac{11}{27},\quad P_4(-1/3)=\frac1{81},\quad
P_6(-1/3)=\frac{47}{243}.
\end{aligned}\end{equation}
Taking $d=2$ in \eqref{eq:moment-bound}, we obtain
\begin{equation}\label{eq:legendre-tail-bound}
P_n(-1/3)\ge-M(1/3)=-\frac{2063}{6561}>-\frac13,
\quad n\ge7.
\end{equation}
Thus the global minimum is $-1/3$, attained at degrees $1$, $2$, and $5$.
Taking the admissible minimising degree $k=5$, we have
$P_2(-1/3)<P_3(-1/3)$, which gives a negative answer to Question~\ref{prob:first} as written.
Here the least minimising degree is $1$, so the example does not refute the
modified statement in which the least minimising degree is chosen.

The same example also violates \eqref{eq:minimiser-interval}:
\begin{equation}\label{eq:jacobi-counterexample}
R_1^{(1,0)}(x)=\frac{1+3x}{4},\quad
R_4^{(1,0)}(1/2)=-\frac{97}{640}<0<R_4^{(1,0)}(1)=1.
\end{equation}
By \eqref{eq:jacobi-counterexample}, $\xi_1=-1/3<1/2<\xi_4$, so the same degree $k=5$ does not satisfy
$-1/3\in[\xi_4,\xi_5]$.

At a minimising degree, \eqref{eq:contiguous-difference} gives only
\begin{equation}\label{eq:minimiser-signs}
Q_{k-1}(x)\ge0
\quad\text{and}\quad Q_k(x)\le0.
\end{equation}
In the example, these values are $7/135$ and $-16/243$.  The second sign
implies $x\le\xi_k$, but the first does not imply $x\ge\xi_{k-1}$: a
Jacobi polynomial may also be positive to the left of its largest zero.
Thus these two signs do not imply the interval condition.  On $[\xi_{k-1},\xi_k]$, however,
\eqref{eq:contiguous-difference} gives
\begin{equation}\label{eq:initial-descent}
U_0(x)\ge\cdots\ge
U_k(x)\le U_{k+1}(x).
\end{equation}
When $\alpha>0$, the interval condition is also necessary, and
Question~\ref{prob:first} has an affirmative answer.

\begin{proposition}\label{prop:positive-minimisers}
Let $\alpha>0$ and $-1<x<1$. The minimum
$\min_{n\ge0}U_n(x)$ is attained at degree $k\ge1$ if and only if
\begin{equation}\label{eq:positive-minimiser-characterisation}
\xi_{k-1}\le x\le\xi_k.
\end{equation}
The minimum is attained only at degree $k$ when
$\xi_{k-1}<x<\xi_k$, and only at degrees $k,k+1$ when $x=\xi_k$.
For every minimising degree $k$,
$U_0(x)\ge U_1(x)\ge\cdots\ge U_k(x)$.
\end{proposition}

\begin{proof}
Under the hypotheses of Corollary~\ref{cor:auxiliary-upper-bound},
write $x=\cos\theta$ and suppose that $\lambda>1/2$ and $x<1$.
We first prove that $W_N^{(c,\lambda)}(x)<1$ for every $N\ge1$.
With the notation of Proposition~\ref{prop:positive-connection},
\eqref{eq:monic-coefficients} gives
$c_{2,0}=\eta_1(1/2)-\eta_1(\lambda)>0$.
Since $\beta_0>0$, \eqref{eq:connection-nonnegative-step} gives
$c_{n+1,0}\ge\beta_0c_{n,0}$ for $n\ge2$, so $c_{N,0}>0$ for $N\ge2$.
By \eqref{eq:connection-induction},
$c_{N,1}\ge c_{N-1,0}>0$ for $N\ge3$.
Thus $d_{N,1}>0$ in \eqref{eq:positive-connection}.
Since $W_1^{(K,1/2)}(x)<1$, Lemma~\ref{thm:Legendre-plateau}
and the positive connection formula give $W_N^{(c,\lambda)}(x)<1$
for $N\ge3$. The case $N=1$ follows from \eqref{eq:W-recurrence}.
For $N=2$, the same recurrence gives
\begin{equation}\label{eq:strict-second-auxiliary}
1-W_2^{(c,\lambda)}(x)
=\frac{2(c+\lambda+1)(1-x)
\bigl(2(c+\lambda)x+c+2\lambda\bigr)}
{(c+2\lambda)(c+2\lambda+1)}>0.
\end{equation}

Let $k\ge3$ and $x\in[\xi_{k-1},\xi_k]$.
As in Section~\ref{sec:largest-zero}, the bound applies with $c=k$
and $c=k+1$. Substitution in \eqref{eq:representation-difference}
gives $U_m(x)>U_k(x)$ for $m\ge k+2$, and for $m=k+1$ unless $x=\xi_k$.
Interlacing and \eqref{eq:contiguous-difference} show that the preceding
degrees are also larger than $U_k(x)$, except for $U_{k-1}=U_k$
at $x=\xi_{k-1}$.

For the first two intervals, put $d=2\alpha+2>2$.
The proof of Lemma~\ref{lem:large-initial-degrees} gives $F(t)>0$
for $t_0\le t<1$, by concavity and \eqref{eq:concavity-endpoints},
and $G(x^2)>0$ throughout the second interval.
Thus the comparisons are strict for $n\ge7$.
For degrees at most six, we use Appendix~\ref{app:initial-degrees}.
In the first interval, \eqref{eq:first-fourth-parameter-derivative}
gives $h_d>h_2\ge0$, the last factor for degree five is positive
by \eqref{eq:first-fifth-endpoint}, and the estimates for degree six
are strict. In the second interval, the comparisons for degrees four
and six are strict. For degree five, we use
\eqref{eq:second-fifth-endpoints} when $x\ge0$. When $x<0$,
\eqref{eq:second-fifth-negative-endpoint} is strict for $2<d\le3$,
and \eqref{eq:second-fifth-negative-large-d} is strict for $d>3$.
The remaining factorisations give equality only between degrees $1,2$
at $-t_0$ and degrees $2,3$ at $t_1$.
This proves the assertions about the minimising degrees.
The final assertion follows from \eqref{eq:initial-descent}.
\end{proof}

\appendix
\section{Two exceptional auxiliary degrees}\label{app:exceptional}

This appendix supplies the two finite-degree estimates used in
Lemma~\ref{thm:Legendre-plateau}.

\begin{lemma}\label{lem:exceptional-auxiliary-degrees}
Let $0<K\le6$ and write $q_n(x)=W_n^{(K,1/2)}(x)$. Then
$q_4(x)\le1$ for $0\le x\le1$, and $q_5(x)\le1$ for
$0\le x\le1/2$.
\end{lemma}

\begin{proof}
The defining recurrence \eqref{eq:W-recurrence} gives
\begin{equation}\label{eq:exceptional-factorisations}\begin{aligned}
1-q_4(x)&=\frac{(2K+5)(1-x)H_4(K,x)}
{(K+1)(K+2)(K+3)(K+4)},\\[7pt]
1-q_5(x)&=\frac{(2K+5)(1-x)H_5(K,x)}
{(K+1)(K+2)(K+3)(K+4)(K+5)}.
\end{aligned}\end{equation}
\smallskip
\noindent\emph{Degree $4$.}
Write
\begin{equation}\label{eq:H4-expansion}
H_4(K,x)=h_3K^3+h_2K^2+h_1K+h_0.
\end{equation}
The four coefficients are
\begin{equation}\label{eq:H4-coefficients}\begin{aligned}
h_3&=2x(4x^2+2x-1),\\[7pt]
h_2&=44x^3+24x^2-9x+1,\\[7pt]
h_1&=62x^3+41x^2-7x+4,\\[7pt]
h_0&=21x^3+21x^2+3x+3.
\end{aligned}\end{equation}
The last three polynomials are positive for $x\ge0$.  Indeed,
\begin{equation}\label{eq:H4-positive-coefficients}\begin{aligned}
h_2&=44x^3+(24x^2-9x+1),\\[7pt]
h_1&=62x^3+(41x^2-7x+4),\\[7pt]
h_0&=3(x+1)(7x^2+1),
\end{aligned}\end{equation}
and the two quadratic factors have negative discriminants.  If
$h_3\ge0$, then $H_4(K,x)>0$.  If $h_3<0$, the two zeros of the derivative
of $H_4(K,x)$ with respect to $K$ have opposite signs.  Hence $H_4(K,x)$ first
increases and then decreases, and its minimum on $0\le K\le6$ occurs at
an endpoint.  Now $H_4(0,x)=h_0>0$ and
\begin{equation}\label{eq:H4-at-six}
H_4(6,x)=3f(x),
\quad f(x)=1235x^3+665x^2-265x+21.
\end{equation}
Put $u=7x-1$.  Then $u\ge-1$ and
\begin{equation}\label{eq:H4-shifted}
343f(x)=1235u^3+8360u^2+30u+108.
\end{equation}
This is positive if $u\ge0$.  If $u=-t$, $0\le t\le1$, then
\begin{equation}\label{eq:H4-negative-shift}
343f(x)=t^2(8360-1235t)+108-30t
\ge7125t^2+78>0.
\end{equation}
Thus $H_4(K,x)>0$.

\medskip

\noindent\emph{Degree $5$.}
Assume $0\le x\le1/2$.  Write
$H_5(K,x)=K^4g_4+K^3g_3+K^2g_2+Kg_1+g_0$, where
\begin{equation}\label{eq:H5-coefficients}\begin{aligned}
g_4&=(2x-1)(2x+1)(4x^2+2x-1),\\[7pt]
g_3&=160x^4+84x^3-76x^2-19x+11,\\[7pt]
g_2&=3g_3+40x^4+46x^3+2x^2-x+8,\\[7pt]
g_1&=3g_3+120x^4+159x^3+19x^2-5x+25,\\[7pt]
g_0&=g_3+29x^4+105x^3+55x^2-2x+13.
\end{aligned}\end{equation}
Put $v=5x-2$.  Since $-2\le v\le1/2$,
\begin{equation}\label{eq:H5-g3-shift}
125g_3=32v^4+340v^3+892v^2+37v+89>0.
\end{equation}
Indeed, the sign is clear for $v\ge0$; if $v=-t$, $0\le t\le2$, the
right-hand side is at least $32t^4+212t^2+15$.  The formulas in \eqref{eq:H5-coefficients}
then show that $g_0,g_1,g_2$ are also positive.  If $g_4\ge0$, then
$H_5(K,x)>0$.  If $g_4<0$, then, for $K>0$,
\begin{equation}\label{eq:H5-K-derivative}
K^{-3}\frac{\partial H_5(K,x)}{\partial K}
=4g_4+\frac{3g_3}{K}+\frac{2g_2}{K^2}+\frac{g_1}{K^3}
\end{equation}
decreases strictly from $+\infty$ to $4g_4<0$.  Hence $H_5(K,x)$ increases
and then decreases with $K$, so its minimum on $0\le K\le6$ is attained
at $K=0$ or $K=6$.  The first value is $g_0>0$, while
\begin{equation}\label{eq:H5-at-six}
H_5(6,x)=3p(x),
\quad
p(x)=25935x^4+13965x^3-12065x^2-3059x+1840.
\end{equation}
The condition $g_4<0$ gives $1/4\le x\le1/2$, so
$-3/4\le v\le1/2$.  In terms of $v$,
\begin{equation}\label{eq:H5-shifted}
125p(x)=5187v^4+55461v^3+147953v^2+15789v+30462.
\end{equation}
This is positive if $v\ge0$.  If $v=-t$, $0\le t\le3/4$, then
\begin{equation}\label{eq:H5-negative-shift}
125p(x)=5187t^4+t^2(147953-55461t)+30462-15789t>0.
\end{equation}
Thus $H_5(K,x)>0$.

The claimed bounds now follow from \eqref{eq:exceptional-factorisations}:
the denominators are positive and $1-x\ge0$ on the respective intervals.
\end{proof}

\section{Low degrees in the first two intervals}\label{app:initial-degrees}

We complete the comparisons of Section~\ref{sec:first-intervals}
for degrees at most six, using only the recurrence
\eqref{eq:ultraspherical-recurrence}. Throughout,
$d=2\alpha+2\ge2$, $t_0=1/(d+1)$, and
$t_1=1/(1+\sqrt{d+4})$, as in \eqref{eq:initial-thresholds}.

\begin{lemma}\label{lem:first-cell}
For $-1\le x\le-t_0$ and every integer $0\le n\le6$,
\begin{equation}\label{eq:first-interval-bound}
U_n(x)\ge U_1(x)=x.
\end{equation}
\end{lemma}

\begin{proof}
Write $x=-t$ and put $G_n(t)=U_n(-t)+t$.
The recurrence \eqref{eq:ultraspherical-recurrence} gives
\begin{equation}\label{eq:first-low-degree-differences}\begin{aligned}
G_0&=1+t,\quad G_1=0,\quad
G_2=\frac{(1+t)((d+1)t-1)}d,\\[7pt]
G_3&=\frac{t(d+3)(1-t^2)}d,\\[7pt]
G_5&=\frac{t(d+5)(1-t^2)
\bigl((d+7)t^2+d-3\bigr)}{d(d+2)}.
\end{aligned}\end{equation}
The signs of $G_0,G_1,G_2,G_3$ in \eqref{eq:first-low-degree-differences} are immediate.  For $G_5$, the sign is
clear when $d\ge3$.  If $2\le d\le3$, the last factor is increasing in
$t$ and its value at $1/(d+1)$ is
\begin{equation}\label{eq:first-fifth-endpoint}
\frac{(d-2)(d-1)(d+2)}{(d+1)^2}.
\end{equation}
For $n=4$, we have
\begin{equation}\label{eq:first-fourth-factor}
G_4(t)=\frac{(1+t)h_d(t)}{d(d+2)},
\end{equation}
where
\begin{equation}\label{eq:first-fourth-polynomial}
h_d(t)=(d+3)(d+5)(t^3-t^2)+(d-1)(d+3)t+3.
\end{equation}
Since $t(1-t)\le1/4$,
\begin{equation}\label{eq:first-fourth-parameter-derivative}
\frac{\partial h_d}{\partial d}
=t\bigl(2d+2-(2d+8)t(1-t)\bigr)\ge0,
\end{equation}
we have $h_d\ge h_2$, where
\begin{equation}\label{eq:first-fourth-at-two}
h_2(t)=3+5t-35t^2(1-t).
\end{equation}
If $t\le4/5$, then $h_2(t)\ge3-15t/4\ge0$.  If $t\ge4/5$, then
$h_2(t)\ge3+5t-7t^2>0$; the last quadratic is concave and positive
at both endpoints.

Finally, put $N_6(d,t)=d(d+2)(d+4)G_6(t)$.  The recurrence \eqref{eq:ultraspherical-recurrence} gives
\begin{equation}\label{eq:first-sixth-polynomial}\begin{aligned}
N_6(d,t)={}&(d+5)(d+7)(d+9)t^6
-15(d+5)(d+7)t^4\\[7pt]
&+45(d+5)t^2-15+d(d+2)(d+4)t.
\end{aligned}\end{equation}
Differentiating \eqref{eq:first-sixth-polynomial} gives
\begin{equation}\label{eq:first-sixth-parameter-derivative}
\frac{\partial N_6(d,t)}{\partial d}
=t(1+t)\bigl(a^2A(t)+aB(t)+C(t)\bigr).
\end{equation}
Here $a=d-2$, and the three polynomials in \eqref{eq:first-sixth-parameter-derivative} are
\begin{equation}\label{eq:first-sixth-derivative-coefficients}\begin{aligned}
\frac{A(t)}3&=1-t(1-t)-t^3(1-t),\\[7pt]
\frac{B(t)}6&=4-4t(1-t)-9t^3(1-t),\\[7pt]
C(t)&=44+t(1-t)-239t^3(1-t).
\end{aligned}\end{equation}
The bounds $t(1-t)\le1/4$ and
$t^3(1-t)\le27/256$ show that $A,B,C$ in \eqref{eq:first-sixth-derivative-coefficients} are positive.  Thus $N_6$
increases with $d$ by \eqref{eq:first-sixth-parameter-derivative}.

Suppose first that $t\ge1/3$.  Then
\begin{equation}\label{eq:first-sixth-at-two}
N_6(d,t)\ge N_6(2,t)=3(1+t)f(t),
\end{equation}
where
\begin{equation}\label{eq:first-sixth-reduction}
f(t)=21\left(t-\frac5{21}-t^2g(t)\right),
\quad g(t)=(1-t)(11t^2-4).
\end{equation}
The positive maximum of $g$ occurs at a point $r>4/5$ satisfying
$33r^2=22r+4$.  Since $(1-t)(11t-4)$ decreases for $t\ge4/5$,
\begin{equation}\label{eq:first-sixth-maximum}
g(r)=\frac23(1-r)(11r-4)<\frac{16}{25}<\frac23.
\end{equation}
Combining \eqref{eq:first-sixth-reduction} and \eqref{eq:first-sixth-maximum} gives
\begin{equation}\label{eq:first-sixth-quadratic-bound}
\frac{f(t)}{21}\ge t-\frac5{21}-\frac23t^2.
\end{equation}
The polynomial on the right is concave, and its values at $1/3$ and
$1$ are $4/189$ and $2/21$.  Thus $f(t)>0$.

If $t\le1/3$, then $d\ge t^{-1}-1$. Monotonicity in $d$ and \eqref{eq:first-sixth-polynomial} give
\begin{equation}\label{eq:first-sixth-small-t}
N_6(d,t)\ge
\frac{(t-1)(t+1)^2(6t+1)}{t^2}
\bigl(32t^4-20t^3-7t^2+4t-1\bigr).
\end{equation}
The last factor in \eqref{eq:first-sixth-small-t} is negative: its first two terms equal
$4t^3(8t-5)\le0$, while $-7t^2+4t-1<0$.  The first factor is also
negative, and therefore $N_6\ge0$.
\end{proof}

\begin{lemma}\label{lem:second-cell}
For $-t_0\le x\le t_1$ and every integer $0\le n\le6$,
\begin{equation}\label{eq:second-interval-bound}
U_n(x)\ge U_2(x).
\end{equation}
\end{lemma}

\begin{proof}
Put $\rho=\sqrt{d+4}>2$. By \eqref{eq:initial-endpoints},
$U_2(x)<0$ throughout the interval. The case $n=2$ is equality, and
the recurrence \eqref{eq:ultraspherical-recurrence} gives
\begin{equation}\label{eq:second-low-degree-differences}\begin{aligned}
U_3(x)-U_2(x)&=\frac{(x-1)((d+3)x^2+2x-1)}d,\\[7pt]
U_4(x)-U_2(x)&=\frac{(d+5)(x^2-1)((d+3)x^2-1)}
{d(d+2)},\\[7pt]
U_6(x)-U_2(x)&=\frac{(d+7)(x^2-1)L(x^2)}
{d(d+2)(d+4)},
\end{aligned}\end{equation}
where
\begin{equation}\label{eq:second-sixth-factor}
L(y)=(d+5)(d+9)y^2+(d+5)(d-6)y+1-d.
\end{equation}
The cases $n=0,1$ are immediate.  For $n=3$, the interval
$[-t_0,t_1]$ lies between the two roots $-1/(\rho-1)$ and
$1/(\rho+1)$ of the quadratic factor.  For $n=4$, we use
$|x|\le1/\sqrt{d+3}$.  Finally,
$0\le x^2\le\max\{t_0^2,t_1^2\}$ for $n=6$.  The polynomial $L$ is
convex, $L(0)<0$, and
\begin{equation}\label{eq:second-sixth-endpoints}\begin{aligned}
L(t_0^2)
&=-\frac{(d+2)(d^4+d^2+26d-8)}{(d+1)^4}<0,\\[7pt]
L(t_1^2)
&=-\frac{2\rho^2(\rho^3+4\rho^2+\rho-8)}
{(\rho+1)^4}<0.
\end{aligned}\end{equation}
Thus $L(x^2)<0$ throughout the interval, and \eqref{eq:second-low-degree-differences} proves the case $n=6$.

It remains only to consider $n=5$.  The recurrence \eqref{eq:ultraspherical-recurrence} gives
\begin{equation}\label{eq:second-fifth-factor}
U_5(x)-U_2(x)=\frac{(x-1)H(x)}{d(d+2)},
\end{equation}
where
\begin{equation}\label{eq:second-fifth-polynomial}\begin{aligned}
H(x)={}&(d+5)(d+7)(x^4+x^3)
+(d+5)(d-3)x^2\\[7pt]
&-(d+17)x-(d+2).
\end{aligned}\end{equation}
For $x\ge0$, differentiating \eqref{eq:second-fifth-polynomial} gives
$H'''(x)=6(d+5)(d+7)(4x+1)>0$ and $H'(0)<0$.
Thus $H$ first
decreases and then, at most, increases, so it is largest at an endpoint
of $[0,t_1]$.
Both endpoint values are negative.
\begin{equation}\label{eq:second-fifth-endpoints}
H(0)=-d-2,\quad H(t_1)=
-\frac{2(\rho+2)(\rho^2+3)(\rho^2+3\rho+1)}
{(\rho+1)^4}.
\end{equation}

For $x=-u\le0$, put $J(u)=H(-u)$.  If $d\ge3$, then
\begin{equation}\label{eq:second-fifth-negative-large-d}\begin{aligned}
J(u)&\le\frac{(d+5)(d-3)}{(d+1)^2}
+\frac{d+17}{d+1}-(d+2)\\[7pt]
&=-\frac{d(d-3)(d+5)}{(d+1)^2}\le0.
\end{aligned}\end{equation}
Suppose that $2\le d\le3$.  Since $0\le u\le1/(d+1)\le1/3$,
\begin{equation}\label{eq:second-fifth-concavity}
J''(u)=6(d+5)(d+7)u(2u-1)+2(d+5)(d-3)\le0.
\end{equation}
Putting $a=d-2$, we have
\begin{equation}\label{eq:second-fifth-endpoint-derivative}
J'\left(\frac1{d+1}\right)
=\frac{3a^4+49a^3+221a^2+295a+72}{(d+1)^3}>0.
\end{equation}
By \eqref{eq:second-fifth-concavity} and \eqref{eq:second-fifth-endpoint-derivative}, $J$ increases, and
\begin{equation}\label{eq:second-fifth-negative-endpoint}
J(u)\le J\left(\frac1{d+1}\right)
=-\frac{d(d-2)(d-1)(d+2)(d+5)}{(d+1)^4}\le0.
\end{equation}
Hence $H(x)\le0$ on the whole interval, and \eqref{eq:second-fifth-factor} gives $U_5(x)\ge U_2(x)$.
\end{proof}

\section*{Acknowledgements}

The discrete energy used in the proof was inspired by the
Pr\"ufer-amplitude comparison in \cite{Castillo2026Szego}. The authors thank S. Yakubovich for a comment that motivated the proof of Proposition 6.3. 
The first author thanks the Institute of Mathematics and the Ministry of Science and Education of Azerbaijan for their hospitality, which made the writing of this paper possible. The first author was supported by the Portuguese Foundation for Science
and Technology through CMUC, project UID/00324/2025, and through project
2022.00143.\allowbreak CEECIND/\allowbreak CP1714/\allowbreak CT0002.
The second author is supported by the Azerbaijan Science Foundation
under grant no.~AEF-MGC-2024-2(50)-16/02/1-M-02.

\makeatletter
\renewcommand{\@biblabel}[1]{#1.}
\makeatother

\end{document}